\documentclass[11pt]{amsart}
\usepackage{amssymb}
\usepackage{fullpage}
\newcommand{\PP}{\mathbb P}

\newcommand\CC{{\mathbb C}}
\newcommand\NN{{\mathbb N}}
\newcommand\ZZ{{\mathbb Z}}
\newcommand\QQ{{\mathbb Q}}
\newcommand\lrr{\langle}
\newcommand\rrr{\rangle} 
\newcommand\mathfraki{\alpha} 
\newcommand{\codim}{\operatorname{codim}}
\DeclareMathOperator{\sing}{\operatorname{sing}}
\DeclareMathOperator{\rank}{rank}

\begin{document}

\title{Non--factorial nodal complete intersection threefolds}

 \author{S{\l}awomir Cynk}
 \address{Institute of Mathematics, Jagiellonian University,
{\L}ojasiewicza 6, 30-348 Krak\'ow, Poland 
}
\address{Institute of Mathematics of the
Polish Academy of Sciences, ul. \'Sniadeckich 8, 00-956 Warszawa,
P.O. Box 21, Poland}
 \email{slawomir.cynk@uj.edu.pl}
 \author {S{\l}awomir Rams}
\address
{Institute of Mathematics,
Jagiellonian University,
ul. {\L}ojasiewicza 6,
30-348 Krak\'ow,
Poland}

\email{slawomir.rams@uj.edu.pl}
\thanks{Research partially supported by NCN grant no. N N201 608040.}

 \begin{abstract}
 We give a bound on the minimal number of singularities of a nodal projective complete intersection threefold 
which contains a smooth complete intersection surface that is not a Cartier divisor.
\end{abstract}

\maketitle

\newcommand{\XX}{X_{16}}
\newcommand{\Pl}{\Pi}
\newcommand{\reg}{\operatorname{reg}}
\theoremstyle{remark}
\newtheorem{obs}{Observation}[section]
\newtheorem{rem}[obs]{Remark}
\newtheorem{examp}[obs]{Example}
\theoremstyle{definition}
\newtheorem{defi}[obs]{Definition}
\theoremstyle{plain}
\newtheorem{prop}[obs]{Proposition}
\newtheorem{theo}[obs]{Theorem}
\newtheorem{lemm}[obs]{Lemma}
\newtheorem*{conj}{Conjecture}
\newtheorem{cor}[obs]{Corollary}
\newcommand{\ux}{\underline{x}}
\newcommand{\ud}{\underline{d}}
\newcommand{\ue}{\underline{e}}
\newcommand{\mmS}{{\mathcal S}}

\section*{Introduction}

The main subject of this consideration are non-factorial nodal
complete intersections in 
projective space. Our interest in 
nodal varieties is justified by the fact that an ordinary double point
(a node) is the most ubiquitous singularity that appears in analytic, algebraic
geometry and singularity theory. A remarkable property of three-dimensional
nodes is the existence of the so-called small resolution, i.e. an
analytic modification that replaces the singular point with one copy
of projective line. In general a small resolution exists only as an
analytic space. Certain global properties of a nodal variety
(e.g. existence of a K\"ahler small resolution, topological
invariants) are closely related to the structure of its class group of
Weil divisors. 
Locally analytically for every node there  exists  germ of a smooth
surface on the threefold that passes through the singularity. 
Is is a subtle question whether algebraic surfaces with such a
property 
exist on a given threefold.

A variety is called factorial if every Weil divisor on it is
Cartier. In the case of a complete intersection threefold $X$ in projective space
it means that each surface on $X$ is defined by a single homogeneous polynomial. 
 Cheltsov (\cite{cheltsov2}, \cite{cheltsov3}), 
 Cheltsov and Park (\cite{cheltsovpark}), Kosta (\cite{kosta},
 \cite{kosta2}) 
gave some lower bounds on the number of singular points on a
non-factorial threefold hypersurfaces in $\PP_{4}$ and complete  
intersections in $\PP_5$ in terms of degrees of defining equations. \smallskip

For a smooth surface in projective space a generic threefold complete
intersection of hypersurfaces of sufficiently high degrees that
contain the surface is nodal (see e.g. \cite{kapustka}).  
Our goal is to study this construction  in more detail in the case of
a complete intersection surface (Thm.~\ref{thm-jest}). 
 The number of nodes of the considered threefold can be computed using Chern classes (see
 Prop.~\ref{lem-eq-sing-on-the-surface}). Using the above fact and an
 elementary but quite tedious integer inequality
 (Lemma~\ref{lem-sde-properties}) we obtain our main result. 

\vskip3mm
\noindent
{\bf Theorem~\ref{thm-inequality}} 
{\it Let $X \subset \PP_{k+3}(\CC)$ be a nodal complete intersection threefold and let $I(X) = \left\langle f_{1}, \ldots, f_{k}\right\rangle$,
where the sequence $(\deg(f_{1}) \ldots,  \deg(f_{k}))$ is non--decreasing.
Assume that 
\begin{eqnarray*} 
& & \mbox{the set } V(f_{1}, \ldots, f_{s}) \mbox{ is smooth in codimension } 3 \mbox{ for } s \leq (k-1),
\end{eqnarray*}
If  a smooth complete intersection surface is not a Cartier divisor on $X$, 
then 

either 
$$
\# \sing(X) \geq \sum_{1\leq i \leq j \leq k} (\deg(f_i)-1) \cdot (\deg(f_j)-1)
$$

or \[\#\sing(X)\ge 2^{k-1} \text{ and }\quad k \in \{2, 3, 4\},\
   \deg(f_1) = \ldots = \deg(f_k) = 2\ .  \] }

\vskip3mm
\noindent
Motivated by the above result we formulate Conjecture (see p.~\pageref{conj}), which in the case of 
a complete intersection in $\PP_5$ is a
slightly stronger version of \cite[Conjecture 30]{cheltsov}
(see Remark~\ref{cheltsovniemaracji}). 

In the last section of the paper we construct numerous examples which show that all integer sequences
considered in Sect.~\ref{sect-main-result} can be realized geometrically with help of complete intersections
(Thm.~\ref{thm-jest}).

\section{An integer inequality} \label{sect-tech-prel}

Given two non--decreasing sequences of positive integers
$\underline{d}:=(d_1, \ldots, d_k)$ and 
$\underline{e}:=(e_1, \ldots, e_{k+1})$  
we define the following quantity
\begin{equation} \label{eq-def-sde}
\mmS(\underline{d}; \underline{e}) \, := \, e_1 \cdot \ldots \cdot
e_{k+1} \cdot \sum_{1 \leq i \leq j \leq k} (d_i - e_i) \cdot (d_j -
e_{j+1}) \, .
\end{equation}

The goal of this section is to prove an elementary inequality
involving   $\mmS(\cdot; \cdot)$ that will be useful in the sequel.
We work under the additional assumption that $(\underline{d}; \underline{e})$ satisfies the following conditions:
\begin{eqnarray}
d_i \geq e_{i\phantom{+1}} &  \mbox{ for } & i \leq k \, ,  \label{eq-cond-one} \\
d_i < e_{i+1} &  \Rightarrow & (d_1,\ldots, d_{i}) = (e_1,\ldots,
e_{i})  \, ,  \text{ for }1\le i\le k\, ,\label{eq-cond-two} \\
d_i < e_{i+2} &  \Rightarrow & (d_1,\ldots, d_{i})  \subset
(e_1,\ldots, e_{i+1}) \text { for }1\le i\le k-1\label{eq-cond-three} \, ,
\end{eqnarray}
where, by abuse of notation,  we write $(d_1, \ldots, d_i) \subset (e_1, \ldots, e_{i+1})$
iff $(d_1, \ldots, d_i)$ is a subsequence of $(e_1, \ldots, e_{i+1})$.
In particular, the conditions 
\eqref{eq-cond-one}, 
\eqref{eq-cond-two} imply that
each summand in $\mmS(\underline{d}; \underline{e})$ is
non-negative.
Moreover, by direct computation one obtains
\begin{equation} \label{eq-sde-vanishing}
\mmS(\underline{d}; \underline{e}) =  0 \quad \mbox{iff} \quad \underline{d} \subset  \underline{e} \, .
\end{equation}
Furthermore, we have the following elementary observation
\begin{equation} \label{eq-lem14}
\mbox{ if } d_i > e_{i+1} \mbox{ for } i=1,\ldots,k \quad  \mbox{ then
} \quad \mmS(\underline{d}; \underline{e}) \geq  \mmS(\underline{d};
1^{k+1}), 
\end{equation}
where,   to simplify  our notation, for an integer $c \in \NN$ we put
$c^k := \underbrace{c, \ldots,c}_{k\text{ times}}$. Indeed, the
inequalities  $1 \leq e_i \leq e_{i+1} < d_i$ imply that   
\begin{equation} \label{eq-trojmian}
e_i \cdot (d_i - e_i) \geq (d_i - 1) \quad \mbox{and} \quad e_{i+1} \cdot (d_i - e_{i+1})  \geq  (d_i - 1) \,  .
\end{equation}
Now \eqref{eq-lem14} follows from \eqref{eq-trojmian} and \eqref{eq-def-sde}.

After those preparations we can prove the main result of this section.
\begin{lemm} \label{lem-sde-properties}
Let $\underline{d}=(d_1, \ldots, d_k)$,  $\underline{e}=(e_1, \ldots, e_{k+1})$ 
be non--decreasing sequences of positive integers 
such that $d_1 \geq 2$ and the conditions 
\eqref{eq-cond-one}, 
\eqref{eq-cond-two},
\eqref{eq-cond-three}
are fulfilled. 
Then \\
either  \[\mmS(\underline{d}; \underline{e})=0\] \\
or  
\begin{equation} \label{eq-sde-inequality}
\mmS(\underline{d}; \underline{e}) \geq \mmS(\underline{d}; {1}^{k+1})
\end{equation} \\
or \[S(\ud;\ue)=2^{k-1}\ \text{ and }k \in \{2, 3, 4\},\ e_1=e_2=1, \quad
d_1 = \ldots = d_k = e_3 = \ldots e_{k+1}=2.\] 
\end{lemm}
\begin{proof} We will proceed by induction on $k$, the case $k=1$
  follows from  the observations  \eqref{eq-sde-vanishing} and
  \eqref{eq-lem14}.

\noindent
Now, assume that the lemma holds for $k-1$ and fix sequences $\underline{d}
$ and $ \underline{e}$ satisfying the assumptions of the lemma. 
We shall separately consider the case
$d_1 = e_1$. If $(d_2,\ldots, d_k) \subset (e_2,\ldots, e_{k+1})$, then 
$\mmS(\underline{d}; \underline{e})=0$ by
\eqref{eq-sde-vanishing}. Otherwise, 
as the sequences $(d_2,\ldots, d_k)$ and $(e_2,\ldots, d_{k+1})$ satisfy the
assumptions of the lemma, we can apply the induction hypothesis \eqref{eq-sde-inequality}  and
obtain
$$ 
\mmS(\underline{d}; \underline{e}) = d_1 \cdot \mmS(d_2,\ldots,
d_k;e_2,\ldots, e_{k+1}) \geq d_1 \cdot  \mmS(d_2,\ldots, d_k;
\underline{1}_{k}) 
=  \mmS(\ud; {1}^{k}, d_1) \, .
$$
Since the equality  $d_{k}=d_{1}$ would imply
$(\underline{d}\subset\underline{e})$, we have   $d_k > d_1$ and \eqref{eq-lem14} gives
$$
\mmS(\ud; {1}^{k}, d_1)  \geq \mmS(\ud; 1^{k+1}) \, .
$$
This completes the proof in the case $d_1 = e_1$.

It remains to consider the case  $d_1 > e_1$. The condition
\eqref{eq-cond-two} yields $d_i \geq e_{i+1}$ for each $i \leq k$,  
and using  \eqref{eq-trojmian} we obtain
\begin{equation} \label{eq-wymiana}
\mmS(\ud; \ue) \geq  \mmS(\ud;1, e_2, \ldots, e_{k+1}).
\end{equation}
By \eqref{eq-def-sde} we have
\begin{equation} \label{eq-suma} 
\mmS(\ud;1, e_2, \ldots, e_{k+1})
= e_2 \cdot  \ldots \cdot  e_{k+1} \cdot (d_1 - 1) \cdot 
\sum_{j=1}^k (d_j - e_{j+1}) +   \mmS(d_2,\ldots, d_k;e_2,\ldots,
e_{k+1}) \, .
\end{equation}
If $\sum_{j=1}^k (d_j - e_{j+1}) \le 0$, then  the inequalities $d_i \geq e_{i+1}$,
where $i=1,\ldots,k$, yield
 $\ud \subset \ue$. \\
Otherwise, repeated use of  \eqref{eq-trojmian}
gives 
\begin{equation} \label{eq-pskladnik}
e_2 \cdot  \ldots \cdot  e_{k+1} \cdot (d_1 - 1) \cdot  \sum_{j=1}^k (d_j - e_{j+1}) \geq
(d_1 - 1) \cdot (\sum_{j=1}^k d_j - k) \, .
\end{equation}

Now, we are in position to apply the induction hypothesis.

\noindent
Suppose that  $\mmS(d_2,\ldots, d_k;e_2,\ldots, e_{k+1}) \geq  \mmS(d_2,\ldots, d_k;{1}^{k})$.
Then \eqref{eq-wymiana}, \eqref{eq-suma}  and \eqref{eq-pskladnik} yield 
$$
\mmS(\ud; \ue) \geq (d_1 - 1) \cdot (\sum_{j=1}^k d_j - k) +  \mmS(d_2,\ldots, d_k;{1}^{k}) = \mmS(\ud; {1}^{k+1}) \, .
$$

\noindent
Assume, that $\mmS(d_2,\ldots, d_k;e_2,\ldots, e_{k+1})$ vanishes. Then, 
by \eqref{eq-sde-vanishing}, we have 
$$
(d_2,\ldots, d_k) \subset (e_2,\ldots, e_{k+1}). 
$$
Therefore, for $i = 2, \ldots,k$ we have either $d_i = e_i$ or $d_i = e_{i+1}$.  
Since $\ue$ is non--decreasing, we obtain  
$$
d_i \leq e_{i+1} \mbox{ for } i = 2, \ldots, k.
$$ 
On the other hand, the assumption $d_1 > e_1$ combined with  \eqref{eq-cond-two} gives $d_i \geq e_{i+1}$
for $i=1,\ldots,k$. Thus we have  $d_i = e_{i+1}$ for $i=2,\ldots,k$.

\noindent
If $d_{k-1} < d_k = e_{k+1}$, then the condition  \eqref{eq-cond-three} gives  
$(d_1,\ldots, d_{k-1}) \subset (e_1, e_2, d_2, \ldots, d_{k-1}).$ 
Since $d_1 > e_1$, we have $d_1 = e_2$ and $\ud \subset \ue$. Consequently, 
by \eqref{eq-sde-vanishing},
we can restrict our attention to the case  $d_{k-1} = d_k$. 
Repeating this reasoning 
we obtain that either $\mmS(\underline{d}; \underline{e})=0$ or
$$
d_1 = \ldots = d_k = e_3 = \ldots = e_{k+1} \mbox{ and } e_2 < d_1 \, .
$$
Suppose that $d_1 \geq 3$. Using  \eqref{eq-trojmian} we obtain the inequalities 
$$
\mmS(d_1, \ldots, d_1; e_1, e_2, d_1, \ldots, d_1) \geq d_1^{k-1} \cdot (d_1 - 1)^2 \geq \binom{k+1}{2} \cdot (d_1 - 1)^2 =
\mmS(d_1, \ldots, d_1; {1}^{k+1}) \, .
$$
For $d_1 = 2$ we get $e_1 = e_2 = 1$, so 
\[\mmS(2^{k};1,1,2^{k-1})=2^{k-1}\ge
\binom{k+1}2=\mmS(2^{k};1^{k+1})\quad\text{for } k\ge5.\]

Finally, assume that 
$$
(d_2,\ldots, d_k;e_2,\ldots, e_{k+1}) \in \{ (2^{2};1^{2},2),  (2^{3};1^{2},2^{2}),
(2^{4};1^{2},2^{3}) \}  
$$ is one of the exceptional cases. Then $d_{1}=2,\ e_{1}=1$, so
$(\ud,\ue) \in \{ (2^{3};1^{3},2),  (2^{4};1^{3},2^{2}),
(2^{5};1^{3},2^{3}) \}$. By direct computation
\eqref{eq-sde-inequality} holds in these cases.
\end{proof}

\section{Main result} \label{sect-main-result}

In this section we prove the main result of this note

\begin{theo} \label{thm-inequality} 
Let $X \subset \PP_{k+3}(\CC)$ be a nodal complete intersection threefold and let $I(X) = \left\langle f_{1}, \ldots, f_{k}\right\rangle$,
where the sequence $(\deg(f_{1}) \ldots,  \deg(f_{k}))$ is non--decreasing.
Assume that 
\begin{eqnarray} 
& & \mbox{the set } V(f_{1}, \ldots, f_{s}) \mbox{ is smooth in codimension } 3 \mbox{ for } s \leq (k-1),  \label{eq-codimthree} 
\label{eq-smooth-fourfold}
\end{eqnarray}
If  a smooth complete intersection surface is not a Cartier divisor on $X$, 
then 

either 
$$
\# \sing(X) \geq \sum_{1\leq i \leq j \leq k} (\deg(f_i)-1) \cdot (\deg(f_j)-1)
$$

or \[\#\sing(X)\ge 2^{k-1} \text{ and }\quad k \in \{2, 3, 4\},\
   \deg(f_1) = \ldots = \deg(f_k) = 2\ .  \]
\end{theo}

\begin{rem} \label{rem-kufactorial}
The local class group of a node is $\ZZ$ (see \cite{milnor}), so the nodal threefold 
is factorial iff it is $\QQ$-factorial. Consequently, Thm~\ref{thm-inequality} gives a bound for non-$\QQ$-factorial 
complete intersections. For the sake of simplicity in this paper we formulate all results using the notion of factoriality.
\end{rem}

\begin{rem} \label{rem-degrees}
Obviously the equations defining a complete intersection are not
unique, whereas their   degrees are uniquely determined. The latter can be
computed f.i. from the minimal resolution. 
\end{rem}

\noindent
In order to simplify our notation, for a nodal complete intersection threefold $X  \subset \PP_{k+3}(\CC)$ and a smooth 
complete intersection surface $S \subset X$ 
from now on we put  
\begin{eqnarray} \label{eq-ordereddegrees} 
I(X) = \left\langle f_{1}, \ldots, f_{k}\right\rangle  \, \, \, \,
\quad &\mbox{ with }& \quad \deg(f_1) \leq \ldots \leq \deg(f_{k})  \, , 
\label{eq-ordereddegreesI}  \\
I(S) = \left\langle g_{1}, \ldots, g_{k+1}\right\rangle  \quad &\mbox{ with }& \quad \deg(g_1) \leq \ldots \leq \deg(g_{k+1}) \, .
\label{eq-ordereddegreesII} 
\end{eqnarray}

The proof of Thm~\ref{thm-inequality}  will be preceded by a proposition and a lemma.
\begin{prop} {\rm (cf. \cite[Ex.~5]{cynk-konf})} \label{lem-eq-sing-on-the-surface}
Let $X \subset \PP_{k+3}(\CC)$  be a nodal complete intersection threefold
and let $S \subset X$ 
be  a smooth complete intersection surface. Then $X$ has 
exactly 
$$
\mmS(\deg(f_{1}), \ldots, \deg(f_{k}); \deg(g_{1}), \ldots, \deg(g_{k+1}))
$$
nodes on $S$.
\end{prop}
\begin{proof} Let $X := Y_{1} \cap \ldots \cap Y_{k}$, where
  $Y_{i}=V(f_{i})$. Let $\sigma: \, \tilde{\PP}_{k+3} \rightarrow \PP_{k+3} $
be the blow-up of $\PP_{k+3}$ along the smooth surface $S$ and let
$\tilde{X}$
be the strict transform of $X$. Then $\tilde{X}$ is again a complete
intersection, i.e. 
$\tilde{X} = \tilde{ Y_{1}} \cap \ldots  \cap \tilde{Y_{k}}$, where
$\tilde{Y_{i}}$ is the strict transform of $Y_{i}$. Since $\tilde{X}$
is a small resolution of the nodes of $X$ that lie on $S$, we have
$$
\tilde{e}(\tilde{X}) = e(X) + \nu \, ,
$$
where $\nu := \# (\sing(X) \cap S)$ and $e$ (resp. $\tilde{e}$)
denotes the
topological Euler characteristic (resp.
degree of the Fulton-Johnson class) as in  \cite{cynk-konf}. On
the other hand we have
\begin{eqnarray*}
&&e(X) = \tilde{e}(X) + \#\operatorname{sing}(X),\\
&&e(\tilde X)=\tilde{e}(\tilde X) + \#\operatorname{sing}(X)-\nu,
\end{eqnarray*}
because the Milnor number of a node is one. Consequently
$$
 \nu = \tfrac{1}{2} (\tilde{e}(\tilde{X}) - \tilde{e}(X)) \, .
$$
By \cite[Prop.~1]{cynk-konf} the number $(\tilde{e}(\tilde{X}) -
\tilde{e}(X))$
is a polynomial in $\deg(f_{1}), \ldots,\deg(f_{k})$,   $\deg(g_{1}),
\ldots$, $\deg(g_{k+1})$
and equals 
$$
\mmS(\deg(f_{1}),
 \ldots, \deg(f_{k}); \deg(g_{1}), \ldots, \deg(g_{k+1}))
$$
if $\deg(f_{i}) \geq\deg(g_{j})$ for all $1 \leq i \leq k$ and 
$1 \leq j \leq (k+1)$. Consequently, the claim of the lemma holds.  
\end{proof} 
In the next lemma we maintain the notation \eqref{eq-ordereddegreesI}, \eqref{eq-ordereddegreesII}.
\begin{lemm} \label{lem-ineq-sing}
Let $X \subset \PP_{k+3}(\CC)$  be a nodal complete intersection threefold, such that  
\eqref{eq-codimthree} holds, 
and let $S \subset X$ 
be  a smooth complete intersection surface.\\
{\rm (I)} For $j \leq k$ one has the inequality:
\begin{equation*} \label{eq-slabestopnie}
\deg(f_j) \geq \deg(g_j) \, .
\end{equation*}
{\rm (II)} If $\deg(f_i) < \deg(g_{i+1})$ for an integer $i \leq k$, 
then
$$
(\deg(f_1), \ldots, \deg(f_i)) =  (\deg(g_1), \ldots, \deg(g_{i})) \, .
$$
{\rm (III)} If $\deg(f_i) < \deg(g_{i+2})$ for an integer $i \leq (k-1)$, 
then
$$
(\deg(f_1), \ldots, \deg(f_i)) \subset   (\deg(g_1), \ldots, \deg(g_{i+1})) \, .
$$
\end{lemm}
\begin{proof} (I) 
Let $i_0$ be the smallest  integer such that 
$\deg(f_{i_0}) < \deg(g_{i_0})$. 
Observe that $i_0 > 1$. Indeed, otherwise 
$\deg(f_1) < \deg(g_l)$ for $l = 1, \ldots, k+1$, which 
implies $f_1 \notin I(S)$ and contradicts the assumption that $S$ is contained in $X$.

Since $I(X) \subset I(S)$, we have $f_j \in \lrr g_1, \ldots, g_{k+1} \rrr $ for each $j$.
On the other hand, if $i \leq i_0$ then the inequality 
$$
\deg(f_i) \leq \deg(f_{i_0}) < \deg(g_{i_0})
$$
holds, which implies $f_j \in \lrr g_1, \ldots, g_{i_0-1}\rrr$. 
We obtain the inclusion
$$
V(g_1, \ldots,  g_{i_0-1}) \subset V(f_1, \ldots,  f_{i_0}). 
$$
The latter contradicts the  assumption that $X$ 
is a complete intersection threefold 
and completes the proof of part (I). 

\noindent
(II) Assume that $\deg(f_{i}) < \deg(g_{i+1})$.
We repeat the proof of part (I) to obtain
the inclusion
$$
\lrr f_1, \ldots,  f_{i} \rrr \subset \lrr g_1, \ldots,  g_{i} \rrr \, .
$$
Consequently, there exist 
homogeneous
polynomials $h_{l,j}$ such that 
\begin{equation} \label{eq-rozwiniecie-fj} 
f_l = \sum_{j=1}^{i} h_{l,j} g_j  \quad \mbox{ for } l =  1, \ldots, i.
\end{equation}
Obviously, one has $h_{l,j} = 0$ when $\deg(f_l) < \deg(g_j)$.  Otherwise, we have 
\begin{equation} \label{eq-degree-hlj}
\deg(h_{l,j}) = (\deg(f_l) - \deg(g_j)) \, .
\end{equation}
Fix a point $P \in V(g_1, \ldots,  g_{i})$.  Since the surface $S$ is assumed to be smooth, 
the polynomials $g_1, \ldots, g_{i}$ form part of a coordinate system  
around $P$ on $\PP_{k+3}$.
Observe, that  the Jacobi matrix of $f_1, \ldots, f_{i}$ with respect to such
coordinates in the point $P$ becomes
$$
\left[
\begin{array}{cccccc}
h_{1,1} & \ldots & h_{1,i} & 0 & \ldots & 0\\
\vdots &     & \vdots & \vdots &     & \vdots    \\
h_{i,1} & \ldots & h_{i,i} & 0 & \ldots & 0
\end{array}
\right] \, .
$$ 
In particular, 
if we put $H := [h_{l,j}]_{l=1,\ldots, i}^{j=1,\ldots, i}$, then
we have  
\begin{equation} \label{eq-inkluzjaII}
V(g_1, \ldots,  g_{i}, \det(H)) \subset \sing(V(f_1, \ldots, f_{i})).
\end{equation}
Suppose that $\deg(\det(H)) > 0$. Then, $\sing(V(f_1, \ldots, f_{i}))$ has a codimension-one component,
which is impossible by \eqref{eq-codimthree}. Therefore, \eqref{eq-degree-hlj} yields
$$
\sum_{j=1}^i (\deg(f_j) - \deg(g_j)) = \deg(\det(H)) = 0. 
$$
The claim results from part (I) of the lemma.

\noindent
(III) Assume that $\deg(f_{i}) < \deg(g_{i+2})$.
As in the proof of part (I) we obtain
the inclusion
$$
\lrr f_1, \ldots,  f_{i} \rrr \subset \lrr g_1, \ldots,  g_{i+1} \rrr
$$
and the  
homogeneous
polynomials $h_{l,j}$ such that 
\begin{equation} \label{eq-rozwiniecie-fjII} 
f_l = \sum_{j=1}^{i+1} h_{l,j} g_j  \quad \mbox{ for } l =  1, \ldots, i.
\end{equation}
Moreover, either \eqref{eq-degree-hlj} holds or   $h_{l,j} = 0$ when $\deg(f_l) < \deg(g_j)$. 

Put $Z := \{ x \in \PP_{k+3} \, : \, g_1(x) =  \ldots =   g_{i+1}(x) = 0  \mbox{ and }    \rank(H(x)) \leq (i-1)  \}$, where
$H := [h_{l,j}]_{l=1,\ldots, i}^{j=1,\ldots, i+1}$. As in \eqref{eq-inkluzjaII} we have the inclusion 
\begin{equation*} \label{eq-inkluzjaIII}
Z  \subset \sing(V(f_1, \ldots, f_{i})) \, .
\end{equation*}
By \cite[Prop.~17.25]{h} either $Z = \emptyset$ or $\dim(Z) = (k-i)$. Thus the assumption \eqref{eq-codimthree}
yields  $Z = \emptyset$. 
In particular, the set  $V(f_1, \ldots,  f_{i})$ is smooth along $V(g_1, \ldots,  g_{i+1})$, so the latter is a Cartier divisor on
$V(f_1, \ldots,  f_{i})$. By    \cite[Cor.~IV.3.2]{hs-a} $V(g_1, \ldots,  g_{i+1})$ is cut out by a hyperplane from 
$V(f_1, \ldots,  f_{i})$, which completes the proof of (III). 
\end{proof}

After those preparations we can give the proof  of  Thm~\ref{thm-inequality}:

\noindent
\begin{proof}[Proof of Thm~\ref{thm-inequality}]
  We maintain the notation \eqref{eq-ordereddegreesI}, \eqref{eq-ordereddegreesII}.
Observe that without loss of generality we can assume that  $X$ is non-degenerate.

Since $S \subset X$ is not Cartier, the threefold $X$ has a node on $S$. By Prop.~\ref{lem-eq-sing-on-the-surface} we have
$$
\mmS(f_{1}, \ldots, f_{k}; g_{1}, \ldots, g_{k+1}) > 0 \, .
$$
Lemma~\ref{lem-ineq-sing} yields that the non--decreasing sequences 
$(f_{1}, \ldots, f_{k})$, $(g_{1}, \ldots, g_{k+1})$ satisfy the conditions
\eqref{eq-cond-one}, 
\eqref{eq-cond-two},
\eqref{eq-cond-three}. Finally,  the theorem results directly from Lemma~\ref{lem-sde-properties}.  
\end{proof}

Motivated by Thm.~\ref{thm-inequality} we propose the following conjecture:

\begin{conj} \label{conj}
Let $X:= V(f_{1}, \ldots, f_{k}) \subset \PP_{k+3}$ be a nodal
complete intersection threefold such that 
the sequence $(\deg(f_{1}),\dots,\deg(f_{k}))$ is non--decreasing and 
$V(f_{1}, \ldots, f_{i})$ is smooth in codimension three for $i=1,\ldots,k-1$.
If $X$ is not factorial, then either 
\[
\# \operatorname{sing}(X) \geq \sum_{1\leq i \leq j \leq k} (\deg(f_i)-1) \cdot (\deg(f_j)-1)
\]

or \[\#\sing(X)\ge 2^{k-1} \text{ and }\quad k \in \{2, 3, 4\},\
   \deg(f_1) = \ldots = \deg(f_k) = 2\ . \]

\end{conj}

We shall see in the next section that if the above conjecture holds
true then the bound it provides is sharp (see Thm.~\ref{thm-existence}).
For the hypersurface case the conjecture was proved by
I.~Cheltsov (\cite{cheltsov3}). For a complete intersection in
$\PP_{5}$ Cheltsov  formulated similar conjecture (see
\cite[Conj.~30]{cheltsov}). He puts a stronger assumption on $X$,
namely he assumes that the hypersurface $V(f_{1})$ is smooth, whereas
Thm.~\ref{thm-inequality} provides a strong evidence that the bound
holds also if $V(f_{1})$ is allowed to have isolated
singularities. The necessity of an assumption on the singularities of
$V(f_{1})$ has been observed by Cheltsov in the following example.
\begin{examp}[\hbox{\cite[Ex.~29]{cheltsov}}]
  Let $Y_{1}=V(l_{1}l_{3}+l_{2}l_{4})$ be a rank-$4$ quadric in $\PP_{5}$,
  and let $Y_{2}=V(f_{2})$ be a general degree-$n$ hypersurface. Then
  the complete intersection $X:=Y_{1}\cap Y_{2}$ has 
exactly  
  $n$ nodes in $V(l_{1},l_{2},l_{3},l_{4},f_{2})$ 
as its only singularities. The threefold $X$ is not
factorial (as it contains the smooth surface
  $V(l_{1},l_{3},f_{2})$), while $$\#\operatorname{sing}(X)=n<
  n^{2}-n+1 \, , $$ which violates the bound of Thm.~\ref{thm-inequality}.
    
Observe that $Y_{1}$ is singular along the line 
$L := V(l_{1}, \ldots, l_{4})$, and  each singularity of $X$ is
obtained as a 
point where $L$ and $Y_{2}$
meet. The above complete intersection corresponds to
the sequences
$\ud = (2,n)$, $\ue = (1,1,n)$ that do not satisfy the condition \eqref{eq-cond-three}. 
\end{examp}

For a fixed integer $k>2$ and $n \gg k$ the sequences $\ud:=(2^{k-1},n)$ and
$\ue:=(1^{k},n)$ produce (via Prop.~\ref{thm-existence}) similar
examples in $\PP_{k+3}$.

\begin{rem} \label{cheltsovniemaracji}
Observe that Thm.~\ref{thm-inequality} implies that
the above Cheltsov's example with $n=2$ and $f_1$, $f_2$ interchanged (i.e. $V(f_2)$ considered as the first hypersurface)
is the only counterexample to \cite[Conjecture 30]{cheltsov} in which a complete intersection surface fails to be a Cartier divisor.
The example in question is one of the three exceptional cases in Thm.~\ref{thm-inequality} 
(i.e. $k=2$, $e_1 = e_2 = 1$ and $e_3 = d_1 = d_2 = 2$).

\end{rem}

\section{General complete intersections} \label{sect-examples}
 
In this section we prove the existence of nodal complete intersection
threefolds that appear in Thm.~\ref{thm-inequality}. Namely,
we show the following theorem.

\begin{theo} \label{thm-jest}
Let $\ud$, $\ue$ be a pair of  non--decreasing sequences of
positive integers satisfying the conditions
\eqref{eq-cond-one}---\eqref{eq-cond-three}.
Then there exists a smooth
surface $S=V(g_{1},\dots,g_{k+1})$ and a threefold
$X=V(f_{1},\dots,f_{k})$ with $\deg (f_{i})=d_{i}$,
$\deg(g_{j})=e_{j}$ such that  $S \subset X$ and
the threefold $X$ has exactly  $\mmS(\ud;\ue)$ nodes on $S$ as its
only singularities. \\
In particular, if $\mmS(\ud, \ue) > 0 $, then the surface $S$ is not a
Cartier divisor, hence the variety  $X$ is not factorial. 
\end{theo}

In fact we shall show a more explicit statement
(Prop.~\ref{thm-existence}). 
In
order to avoid unnecessary technical complications in its proof, we
represent the polynomial $f_{i}$ as a linear combination of minimal
number of the generators $g_{j}$. To make this explicit, 
for  $i \leq (k-1)$ we define 
$$
\mathfraki_{(\ud;\ue)}(i) :=  \left\{ \begin{array}{cl}
i    & \mbox{ if }  d_i < e_{i+1} \, ,  \\
i+1  & \mbox{ if }    d_i \geq  e_{i+1} \mbox{ and } d_i < e_{i+2} \, ,  \\
i+2  & \mbox{ otherwise. }  
\end{array}
\right.
$$
Whenever it leads to no ambiguity we write $\mathfraki(i) :=
\mathfraki_{(\ud;\ue)}(i)$. Finally,  we put  $\mathfraki(k) := k+1 $.

We choose non-zero homogeneous polynomials  $g_1$, $\ldots$, $g_{k+1}$ such that  $\deg(g_j) = e_j$ and
\begin{equation} \label{eq-int-smooth}
\mbox{V}(g_1, \ldots, g_i) \mbox{ is smooth  for } i \leq k + 1
\end{equation}
Given the polynomials $g_{1}, \ldots, g_{k+1}$ and 
homogeneous polynomials $h_{i,j}$ of degree $(d_i - e_j)$, we define  
 a degree-$d_i$ homogeneous polynomial $f_i \in \lrr g_1, \ldots,
g_{\mathfraki(i)} \rrr$ by the equality  
\begin{equation} \label{eq-sumfi} 
f_i = \sum_{j=1}^{\mathfraki(i)} h_{i,j} g_j  \, .  
\end{equation}

In order to prove  Prop.~\ref{thm-existence} we need several lemmata.
In their proofs we put $h_{i,j} := 0$ for $j > \mathfraki(i)$ and  define the matrix
$$
H_i := [h_{l,j}]_{l=1,...,i}^{j=1,...,\mathfraki(i)} \, .
$$

At first we deal with the smooth case.
\begin{lemm} \label{lem-sing-easycase}
Fix  $g_1$, $\ldots$ $g_{k+1}$ that satisfy \eqref{eq-int-smooth}. If $i \leq k-1$ and 
$\mathfraki(i) \leq  i+1$, then for generic  polynomials  $h_{l,j}$, 
where $l \leq i$ and $j \leq \mathfraki(l)$, the intersection 
$\mbox{V}(f_1, \ldots, f_i) \mbox{ is smooth along } \mbox{V}(g_1, \ldots, g_{\mathfraki(i)})$.
\end{lemm}
\begin{proof}
{\it Case  $\mathfraki(i) = i$:}
As in the proof of Lemma~\ref{lem-ineq-sing}.(II) we have the equality
$$
V(g_1, \ldots,  g_{i}, \det(H_i)) =  \sing(V(f_1, \ldots, f_{i})) \cap V(g_1, \ldots,  g_{i})
$$
and $\det(H_i)$ is a degree-$0$ polynomial. Obviously, for generic  $h_{l,j}$, we have  $\det(H_i) \neq 0$.

\noindent
{\it Case  $\mathfraki(i) = i+1$:} 
As in  the proof of Lemma~\ref{lem-ineq-sing} we have the equality
\begin{equation} \label{eq-drobiazg}
\sing(V(f_1, \ldots, f_{i})) \cap V(g_1, \ldots,  g_{i+1}) = \{ P \in V(g_1, \ldots,  g_{i+1}) \, : \, \rank(H_i(P)) \leq (i-1) \} \, .
\end{equation}
Moreover, by \eqref{eq-cond-three}, there exists an integer $i_0 \geq 1$
such that  $$(e_1, \ldots, e_{k+1}) = (d_1, \ldots, d_{i_0}, e_{i_0+1},  d_{i_0+1}, \ldots, d_{k}),$$ 
so 
the polynomials $h_{1,1}, \ldots, h_{i_0,i_0}$, $h_{i_0+1,i_0+2}, \ldots, h_{i,i+1}$ are constant and 
determinant of the matrix obtained from $H_i$ by deleting its $i_0{}^{th}$ column is a degree-$0$ polynomial.
For generic  $h_{l,j}$ the determinant in question 
does not vanish and the lemma follows. 
\end{proof}

\begin{lemm} \label{lem-sing-caseone}
 Assume that $i \leq (k-1)$ and 
$\mathfraki(i) \leq  (i+1)$.
For generic choice of  the polynomials $g_1$, $\ldots$ $g_{\mathfraki(i)}$ and  $h_{l,j}$, 
where $l \leq i$ and $j \leq \mathfraki(l)$ the intersection $\mbox{V}(f_1, \ldots, f_i)$ is smooth. 

\end{lemm}
\begin{proof} We proceed by induction on $i$.

\noindent
[i=1]:  If $\mathfraki(1) = 1$, then $f_1 = h_{1,1} \cdot g_1$ with $\deg(h_{1,1}) = 0$, so $V(g_1)= V(f_1)$ is smooth by assumption.

\noindent
Suppose that $\mathfraki(1) = 2$, then $f_1 = h_{1,1} \cdot g_1 + h_{1,2} \cdot g_2$ with $\deg(h_{1,2}) = 0$. 
Fix $g_1$, $g_2$ satisfying \eqref{eq-int-smooth}. 
By Bertini, for generic $h_{1,1}, h_{1,2}$ the hypersurface $V(f_1)$ is smooth away from $V(g_1, g_2)$. 
Therefore, generic $V(f_1)$ is smooth by Lemma~\ref{lem-sing-easycase}.

\noindent
[(i-1) $\leadsto$ i]: 
If $\mathfraki(i-1) \leq i$, then by the induction hypothesis,
for generic $g_1, \ldots, g_{i}$ and $h_{1,1}, \ldots, h_{i-1,i}$
the intersection $\mbox{V}(f_1, \ldots, f_{i-1})$ is smooth.
By Bertini, for generic $h_{i,1}, \ldots, h_{i,i}$ the intersection  $V(f_1, \ldots, f_i)$ is smooth away from $V(g_1, \ldots, g_{i})$. 
Lemma~\ref{lem-sing-easycase} completes the proof in this case.

\noindent
Suppose that  $\mathfraki(i-1) = i+1$. By definition, we have $d_{i-1} \geq e_{i+1}$.
From $\mathfraki(i) = i+1$ and \eqref{eq-cond-three} we infer that $d_i \in \{e_i, e_{i+1}\}$,
so $d_i \leq e_{i+1}$ and $d_{i-1} = d_i = e_{i+1}$. Finally, from \eqref{eq-cond-three} we obtain
$d_{i-1} \in \{e_{i-1}, e_{i}\}$, so $d_{i-1} \leq  e_i  \leq  e_{i+1}$ and  $d_{i-1} = e_{i}$.
In particular, we have the inclusion 
$$
(d_1, \ldots, d_{i-1}) \subset (e_1, \ldots, e_{i}).
$$
Consequently, if we define $\tilde{\ud} := (d_1, \ldots, d_{i-1}, d_{i-1}+1, \ldots,  d_{i-1} + k - i + 1)$
and $\tilde{\ue} := (e_1, \ldots, e_{i}, d_{i-1}+1, \ldots,  d_{i-1} + k - i + 1)$,
then the pair  $(\tilde{\ud}, \tilde{\ue})$ satisfies the conditions   \eqref{eq-cond-one}, 
\eqref{eq-cond-two},
\eqref{eq-cond-three}. By definition,  for $l < i-1$
\begin{equation*} \label{eq-zmiana}
\mathfraki_{(\tilde{\ud};\tilde{\ue})}(l)=\mathfraki_{(\ud;\ue)}(l) \mbox{ and }  \mathfraki_{(\tilde{\ud};\tilde{\ue})}(i-1) = i  
\end{equation*}
so, by the induction hypothesis,  $\mbox{V}(\tilde{f}_1, \ldots, \tilde{f}_{i-1})$ is smooth if we make generic choices
of $\tilde{h}_{l,j}$, 
where $l \leq i-1$ and $j \leq \mathfraki_{(\tilde{\ud};\tilde{\ue})}(l)$. 
Assume the latter holds, and put 
$h_{l,j} = \tilde{h}_{l,j} \mbox{ for } l \leq i-1 \mbox{ and } j \leq \mathfraki_{(\tilde{\ud};\tilde{\ue})}(l)$.
Then for generic choice of $h_{i-1,i+1} \in \CC$ the intersection $\mbox{V}(f_1, \ldots, f_{i-1})$ is smooth.
Finally, Bertini and Lemma~\ref{lem-sing-easycase} imply that $\mbox{V}(f_1, \ldots, f_{i})$ is smooth for generic choices
of $h_{i,1}, \ldots, h_{i,i+1}$. 
\end{proof}

Now we can deal with the general case.
\begin{lemm} \label{lem-sing-gencase}
Let $i \leq (k-1)$.
For generic choice of  the polynomials $g_1$, $\ldots$ $g_{\mathfraki(i)}$ and  $h_{l,j}$, 
where $l \leq i$ and $j \leq \mathfraki(l)$, the intersection
$\mbox{V}(f_1, \ldots, f_i)$ is smooth in codimension four.
\end{lemm}
\begin{proof} We proceed by induction on $i$.

\noindent
[i=1]:  If $\mathfraki(1) \leq 2$, then generic $V(f_1)$ is smooth by Lemma~\ref{lem-sing-caseone}.
Suppose that $\mathfraki(1) = 3$, then 
$$
f_1 = h_{1,1} \cdot g_1 + h_{1,2} \cdot g_2 +  h_{1,3} \cdot g_3 \, .
$$
If \eqref{eq-int-smooth} holds,  we have  $\sing(V(f_1)) \subset V(g_1, g_2, g_3,  h_{1,1}, h_{1,2}, h_{1,3})$
for generic $h_{1,1}, \ldots, h_{1,3}$.

\noindent
[(i-1) $\leadsto$ i]: By Lemma~\ref{lem-sing-caseone} we can assume 
$\mathfraki(i) = i+2$. 
We apply  Bertini on $\reg(\mbox{V}(f_1, \ldots, f_{i-1}))$ to obtain the inclusion
$$
\sing(\mbox{V}(f_1, \ldots, f_i)) \subset \mbox{V}(g_1, \ldots, g_{i+2}) \cup \sing(\mbox{V}(f_1, \ldots, f_{i-1})) \, .
$$
For generic choice of $g_{i+2}$ no component of
$\sing(\mbox{V}(f_1, \ldots, f_{i-1}))$  is contained in   $V(g_{i+2})$. By the induction hypothesis,  
it remains to study the components of $\sing(\mbox{V}(f_1, \ldots, f_{i}))$ that are contained in 
$\mbox{V}(g_1, \ldots, g_{i+2})$.

Recall that we assumed \eqref{eq-int-smooth}, so we have the equality
$$
\sing(V(f_1, \ldots, f_{i})) \cap V(g_1, \ldots,  g_{i+2}) = \{ P \in V(g_1, \ldots,  g_{i+2}) \, : \, \rank(H_i(P)) \leq (i-1) \} \, . 
$$
In particular, the induction hypothesis implies that  for generic $g_{i+2}$ the set
\begin{equation} \label{eq-rank-iminus2}
\{ P \in V(g_1, \ldots,  g_{\mathfraki(i-1)}) \, : \, \rank(H_{i-1}(P)) \leq (i-2) \}
\end{equation}
has codimension at least five in $V(f_1, \ldots, f_{i-1}, f_i)$. Thus, 
we can restrict our attention to the components
of the singular locus that meet the 
set $\{ P \,: \,   \rank(H_{i-1}(P)) = (i-1) \}$. 

Observe that $\mathfraki(i-1) < i+2$, so the  components in question are contained in $V(h_{i,i+2})$.
If $\deg(h_{i,i+2}) = 0$ we can choose a nonzero $h_{i,i+2}$ and the proof is complete.
Otherwise,  we can assume that the hypersurface  $V(h_{i,i+2})$ meets $V(g_1, \ldots,  g_{i+2})$ properly. 

Let $\hat{H}_{i-1} := [h_{l,j}]^{j=1, \ldots,i-1}_{l=1, \ldots,i-1}$. 
Determinant of the minor of $H_i$ obtained by removing the  $i^{th}$ and $(i+2)^{nd}$ column can be written as
\begin{equation} \label{eq-drugiminor}
h_{i,i+1} \cdot \det(\hat{H}_{i-1}) + h_{i-1,i+1} \cdot \det(M_1), 
\end{equation}
where $M_1$ is another  minor of $H_{i}$. One can easily see, that  for generic 
$h_{i,1}, \ldots, h_{i,i-1}$ and  $h_{i,i+1}$ the hypersurface given by  \eqref{eq-drugiminor}
meets $V(g_1, \ldots,  g_{i+2}, h_{i,i+2})$ properly. 
An analogous argument applied to 
the minor  obtained by removing the  $(i+1)^{st}$  and $(i+2)^{\rm nd}$  column from $H_i$
shows that the 
$$
\codim(\sing(V(f_1, \ldots, f_i)) \cap  \{ P \,: \,   \rank(H_{i-1}(P)) = (i-1) \}) \geq 5,
$$ 
and it remains to deal with the components of $\sing(V(f_1, \ldots, f_i))$ that are contained in the set
\begin{equation} \label{eq-last-step}
V(g_1, \ldots,  g_{i+2}, h_{i,i+2}, \det(\hat{H}_{i-1})) \, .
\end{equation}
If $\mathfraki(i-1)=i-1$, then 
$\det(\hat{H}_{i-1})$ is constant and we are done. Otherwise, we can
assume that  \eqref{eq-last-step} is of codimension four and
 use another minor of the matrix $H_{i-1}$
to complete the proof. 
\end{proof}

\begin{lemm} \label{lem-A1}
Assume that  $g_1$, $\ldots$ $g_{k+1}$ satisfy the condition \eqref{eq-int-smooth}  and 
the fourfold
$V(f_1, \ldots, f_{k-1})$ is smooth.
For generic $f_k \in \mathcal{O}_{\PP_{k+3}}(d_k) \otimes \mathcal{I}(S)$ the intersection 
$V(f_1, \ldots, f_k)$ has at most nodes (A$_1$ singularities).
\end{lemm}
\begin{proof}
Put $Y := V(f_1, \ldots, f_{k-1})$. Bertini theorem implies that  generic $\mbox{V}(f_k)$ is smooth away from the surface $S$.
We cover the surface in question  with the open sets
$$
U_{i,j} := \{ P \in S \, : \, \mathcal{I}(S) \mathcal{O}_{Y,P} = \lrr g_i, g_j \rrr \mathcal{O}_{Y,P} \} \, ,
$$ 
where $i,j = 1, \ldots, k+1$. By stability property of A$_1$ points, the set
$$
\mathcal{L}_{i,j} := \{ f_k \in \mathcal{O}_{\PP_{k+3}}(d_k) \otimes \mathcal{I}(S) \, : \; V(f_k|_{Y}) 
\mbox{ has at most nodes on } U_{i,j} \}
$$
is open in $\mathcal{O}_{\PP_{k+3}}(d_k) \otimes \mathcal{I}(S)$.
To show that $\mathcal{L}_{i,j}$ is non-empty observe that for generic degree-$(d_k - e_i)$ (resp. degree-$(d_k - e_j)$)  homogeneous polynomial  $h_{k,i}$
(resp. $h_{k,j}$) the threefold given on $Y$ by   vanishing of 
\begin{equation} \label{eq-lij}
h_{k,i} \cdot g_i + h_{k,j} \cdot g_j
\end{equation}
has nodes in $V(g_i, g_j, h_{k,i}, h_{k,j})$  as its only singularities. 
In particular, the polynomial \eqref{eq-lij} belongs to $\mathcal{L}_{i,j}$.
\end{proof}

Now we are in position to
formulate and prove the following proposition that directly implies Thm.~\ref{thm-jest}.  
In Prop.~\ref{thm-existence} we maintain the notation  \eqref{eq-sumfi}.

\begin{prop} \label{thm-existence} Let $\underline{d}:=(d_1, \ldots, d_k)$ and 
$\underline{e}:=(e_1, \ldots, e_{k+1})$  
 be a pair of  non--decreasing sequences of positive integers
that 
satisfy the conditions \eqref{eq-cond-one}, \eqref{eq-cond-two}, \eqref{eq-cond-three}. 
For generic choice of  degree-$e_j$ homogeneous  polynomials $g_j$
and degree-$(d_i - e_j)$ homogeneous polynomials $h_{i,j}$, 
where $j = 1, \ldots, k+1$ and $i = 1, \ldots, \mathfraki(j)$, 

\noindent
\begin{itemize}
\item [(I)]  the intersections  $V(f_{1}, \ldots, f_{s})$ are smooth
  in codimension three for  $s \leq k-1$,
\item [(II)] the threefold   $V(f_{1}, \ldots, f_{k})$  has exactly 
$\mmS(\underline{d}; \underline{e})$ nodes on the surface $S$  as its
only singularities.

\item [(III)] If  $\mmS(\underline{d}; \underline{e}) > 0$, then the surface $S$ is
not a Cartier divisor on the threefold $V(f_{1}, \ldots, f_{k})$.

\end{itemize}
\end{prop}
\begin{proof}
Claim (I) follows directly from Lemma~\ref{lem-sing-gencase}. In particular, 
Lemma~\ref{lem-sing-gencase} implies that the fourfold $V(f_1, \ldots, f_{k-1})$ is smooth and we can apply 
Lemma~\ref{lem-A1} to show that all singularities of the threefold  $V(f_1, \ldots, f_k)$ are  nodes. 
In this way part (II) of the theorem is also proven.

Suppose that $\mmS(\underline{d}; \underline{e}) > 0$. By (I) and (II) the assumptions
of Prop.~\ref{lem-eq-sing-on-the-surface} are fulfilled. Consequently, the threefold 
has exactly $\mmS(\underline{d}; \underline{e})$ nodes on the surface $S$ and the
latter cannot be a Cartier divisor.  
\end{proof}

\end{document}